\documentclass[preprint]{elsarticle}
\usepackage{cmap} 
           \usepackage[T2A]{fontenc}   
           \usepackage[utf8]{inputenc}
           \usepackage[english,russian]{babel}
\usepackage{amsmath}
\usepackage{amsfonts}
\usepackage{amssymb}
\usepackage{amsthm}
\usepackage[pdftex,unicode]{hyperref}
\usepackage{indentfirst} 
\usepackage{amscd}
\usepackage[all,cmtip]{xy}
\usepackage{mathtools}
\usepackage{comment}
\usepackage{tikz-cd}
\usetikzlibrary{babel}

\newtheorem{proposition}{Предложение}
\newtheorem{theorem}{Теорема}
\newtheorem*{theorem*}{Теорема}
\newtheorem{lemma}{Лемма}
\newtheorem*{lemma*}{Лемма}
\newtheorem{cor}{Следствие}

\theoremstyle{definition}
\newtheorem{example}{Пример}

\newtheorem{problem}{Проблема}

\theoremstyle{remark}

\newtheorem*{note*}{Замечание}

  \def\R{\mathbb R}

\def\N{\mathbb N}

\def\om{\omega}
\def\Om{\Omega}

\def\al{\alpha}
\def\ga{\gamma}
\def\la{\lambda}
\def\ph{\varphi}
\def\de{\delta}
\def\ka{\kappa}

\def\nfs/{NFS}
\def\cdp/{CDP}
\def\cdpz/{CDP${}_0$}

\def\cN{\mathcal{N}}
\def\cL{\mathcal{L}}

\def\ds#1{\D(#1)}

\def\cl#1{\overline{#1}}
\def\clx#1#2{\overline{#2}^{\,#1}}

\def\bt{\operatorname{\beta}}

\def\Tau{{\mathcal T}}

\def\sset#1{\{#1\}}

\def\set#1{\bbset#1\eeset}
\def\bbset#1:#2\eeset{\{#1\,:\,#2\}}

\def\bbsett#1:#2\eesett{\{#1\,:\,\text{#2}\}}

\def\ibbset#1:#2\ieeset{(#1)_{#2}}

\def\cM{{\mathcal M}}

\def\tp{\Tau}

\def\cP{{\mathcal P}}
\def\cB{{\mathcal B}}

\def\cB{{\mathcal B}}

\def\cU{{\mathcal U}}

\def\cR{{\mathcal R}}

\def\cV{{\mathcal V}}
\def\cN{{\mathcal N}}

\def\diag{\mathop{\bigtriangleup}}

\newcommand\restrA[2]{{
  \left.\kern-\nulldelimiterspace 
  #1 
  \vphantom{\big|} 
  \right|_{#2} 
  }}

\newcommand\restrB[2]{\ensuremath{\left.#1\right|_{#2}}}

\def\restr#1#2{\restrB{#1}{#2}}

\def\pwr#1_#2{#1^{[#2]}}

\def\term#1{{\it #1}}

\def\dddm#1(#2){N_{#1}(#2)}
\def\dddb#1(#2){B_{#1}(#2)}

\def\et(#1){ (#1)}

\def\bitem#1,#2.{ $#2\nrightarrow #1$:\ }

\def\oo#1/{$O_{#1}$}

\def\ep{\varepsilon}

\def\gd/{$G_\delta$}

\def\F{\Phi}

\def\rarr{\Rightarrow}

\def\diag{\DD}

\def\si{\sigma}

\def\ext#1{{\mathcal E}(#1)}
\def\af#1{{\mathcal A}(#1)}
\def\afb#1{{\mathcal A}_1(#1)}

\def\nm#1{\lVert #1 \rVert_{\infty}}
\def\nmc{\nm{\cdot}}

\def\csr#1{{\mathcal R}_{#1}}

\def\cz{$cz$}
\def\nwcz{nw_{cz}}

\def\ds{d_s}
\def\hds{hd_s}
\def\Si{\Sigma}
\def\:{\colon}

\def\diag{\mathop{\bigtriangleup}}

\def\leq{\leqslant}
\def\geq{\geqslant}

\def\cc{{\mathfrak{c}}}

\def\dc{D_{\cc}}
\def\dcc{D_{\cc,\om}}
\def\kc{K_{\cc}}

\begin{document}

\begin{frontmatter}

\title{Вес выпуклых компактов и их границ}
\author{Евгений Резниченко}

\ead{erezn@inbox.ru}

\address{Department of General Topology and Geometry, Mechanics and  Mathematics Faculty, M.~V.~Lomonosov Moscow State University, Leninskie Gory 1, Moscow, 199991 Russia}

\begin{abstract}
Доказано, что если некоторая граница $B$ выпуклого компактного подмножества $X$ локально выпуклого линейного пространства имеет счетную сеть, то выпуклый компакт $X$ метризуем. Если граница $B$ линделефовово $\Si$-пространство, то сетевой вес $nw(B)$ границы $K$  совпадает с весом $w(K)$ выпуклого компакта $K$.
\end{abstract}
\end{frontmatter}

\def\Q{\mathbb Q}
\def\Z{\mathbb Z}

\section{Введение}
\label{sec:intro}

Пусть $K$ выпуклое компактное подмножество некоторого локально выпуклого пространства, $\ext K$ множество крайних точек выпуклого компакта $K$. 
В \cite{Reznichenko1990} доказано, что вес $w(\ext K)$ крайних точек совпадает с сетевым весом $nw(\ext K)$ и  получены условия, при которых вес $w(\ext K)$ крайних точек совпадает с весом $w(K)$  выпуклого компакта. Например, когда $\ext K$ линделефово пространство или $K$ является симплексом Шоке. Но ответ на следующий вопрос до сих пор неизвестен.

\begin{problem}[А.В. Архангелький] \label{q:1}
Верно ли, что вес $w(K)$ выпуклого компакта $K$ совпадает с весом $w(\ext K)$ его крайних точек $\ext K$?
\end{problem}


Пусть $B\subset K$ некоторая граница выпуклого компакта $K$. В работах \cite{Spurny2010,KalendaSpurny2009} исследуется, как связаны вес $w(K)$ выпуклого компакта, вес $w(B)$ и сетевой вес $nw(B)$ границы $B$ и каким образом результаты о крайних точках  можно перенести на  границы.

В \cite[Example 5.3]{KalendaSpurny2009} построен выпуклый компакт $K$ и граница $B\subset K$, для которых $nw(B)<w(B)=w(K)$. В \cite{KalendaSpurny2009} ставится вопрос: верно ли что вес $w(B)$ границы $B$ выпуклого компакта $K$ совпадает с весом $w(K)$ выпуклого компакта $K$? В этой заметки получен отрицательный ответ на этот вопрос, построен выпуклый компакт $K$ и граница $B\subset K$, так что $w(B)<w(K)$ (Пример \ref{e:eq:1}).
В \cite[Corollary 1.4]{Spurny2010} И.Спурны доказал, что если граница $B$ линделефова, то $w(K)=w(B)$. В \cite[Question 1.5]{Spurny2010} ставится следующий вопрос.

\begin{problem}[И.Спурны] \label{q:2}
Пусть $K$ выпуклый компакт,  $B$ граница в $K$ и $B$ линделефово. 
Верно ли, что $w(K)=nw(B)$?
\end{problem}
Если $B=\ext K$, то ответ на этот вопрос положителен \cite[Theorem 3.3]{Reznichenko1990}.
В этой работе дается частичный ответ на вопрос И.Спурны. 
Пусть $\tau$ кардинал. Определим класс пространств $\csr\tau$: $X\in\csr\tau$ если существует семейство подмножеств $\cN$ множества $X$, так что для каждого нуль множества $U\subset X$ существует $\cL\subset \cN$, $|\cL|\leq \om$, такое что $\bigcup\cL=U$.
Тогда выполняются следующие условия (см.\ теоремы \ref{t:asd:1} и \ref{t:wccs:2}):
\begin{enumerate}
\item
если $K$ выпуклый компакт, $B\subset K$ граница $K$ и $B\in \csr\tau$, то $w(K)\leq\tau$;
\item
если $X$ линделефово пространство и $w(X)\leq\tau$, то $X\in\csr\tau$;
\item
если $Y\in\csr\tau$ и $X\subset Y$ $z$-вложено в $Y$, то $X\in\csr\tau$;
\item
если $Y\in\csr\tau$ и $X$ является непрерывным образом $Y$, то $X\in\csr\tau$;
\item
если  $X$ является линделефовым $\Si$-пространством и $nw(X)\leq\tau$, то $X\in\csr\tau$;
\item
если  $X$ плотное подмножество произведения $\prod_{\al<\tau} X_\al$, где $X_\al$ метризуемое сепарабельное пространство для $\al<\tau$, то $X\in\csr\tau$.
\end{enumerate}
Получаем, что если граница $B$ является линделефовым $\Si$-пространством, то $w(K)=nw(B)$ (теорема \ref{t:wccs:3}). 
В частности, если $K$ выпуклый компакт и $B\subset K$ граница со счетной сетью, то выпуклый компакт $K$ метризуем (теорема \ref{t:wccs:2+1}). 
В \cite{Haydon1976} Р.~Хейдон доказал это утверждение в случае, когда $B$ является множеством $\ext K$ крайних точек выпуклого компакта $K$. 
Из (1) и (6) вытекает, что если граница $B$ гомеоморфна $\R^\tau$, то $w(K)\leq\tau$. Это утверждение интересно и неожиданно и для случая, когда $B=\ext K$.

Раздел \ref{sec:defs} содержит определения, обозначения и предварительные сведения.
В разделе \ref{sec:asd} изучается, когда в подпространствах функциональных пространств есть секвенциально плотные подпространства фиксированной мощности.
В разделе \ref{sec:wcps} изучается связь между весом выпуклых компактов и пространством вероятностных мер над границами.
В разделе \ref{sec:wccs} изучается связь между весом выпуклых компактов и пространством аффинных функций над выпуклым компактом.
В разделе \ref{sec:eq} строятся примеры и формулируются вопросы.


\section{Определения, обозначения и предварительные сведения}
\label{sec:defs}

Натуральные числа обозначаются как $\N$, $\om=\sset 0 \cup \N$ --- первый счетный ординал, $\om_1$ --- первый несчетный ординал. Под пространством мы подразумеваем тихоновское пространство. 
Пусть $X$ пространство.

Семейство $\cN$ подмножеств пространства $X$  называется \term{сетью} пространства $X$ если для любого открытого множества $U\subset X$ существует подсемейство $\cM\subset \cN$, такое что $U=\bigcup\cM$. Семейство $\cN$ называется \term{базой} пространства $X$ если $\cN$ состоит из открытых множеств и является сетью пространства $X$. 

Подмножество $S\subset X$ называется \term{секвенциально плотным} в $X$, если для любого $x\in X$ существует последовательность $(x_n)_n\subset S$, сходящаяся к точке $x$.

Пространство $X$ называется пространством \term{Фреше--Урысона}, если для каждого $A\subset X$ и $x\in \cl X$ существует последовательность $(x_n)_n\subset A$, сходящаяся к $x$.

Пусть $f\: X\to Y$ отображение пространств. Отображение $f$ называется \term{уплотнением}, если $f$ непрерывная биекция. Отображение $f$ называется \term{совершенным}, если $f$ непрерывное замкнутое отображение, такое что $f^{-1}(x)$ компактно для всех $x\in Y$.

Определим некоторые кардинальнве инварианты:
\begin{align*}
iw(X) &= \min \{\ \tau: X \text{ уплотняется на пространство с весом не более } \tau \} + \om,
\\
w(X) &= \min \{\ |\cB|: \cB \text{ является базой пространства } X \} + \om,
\\
nw(X) &= \min \{\ |\cN|: \cN \text{ является сетью пространства } X \}  + \om,
\\
d(X) &= \min \{\ |M|: M\subset X = \cl M \}  + \om,
\\
hd(X) &= \sup \{\ d(M): M\subset X\},
\\
l(X) &= \min \{\ \tau: \text{ для любого открытого покрытия }\ga\text{ пространства }X
\\
&\qquad\qquad\quad
\text{ существует }\la\subset\ga,\text{ такое что }|\la|\leq\tau\text{ и }\bigcup\la=X\ \},
\\
hl(X) &= \sup \{\ l(M): M\subset X\},
\\
hl^*(X) &= \sup \{\ hl(X^n): n\in\om\},
\\
wl(X) &= \min \{\ \tau: \text{ для любого открытого покрытия }\ga\text{ пространства }X
\\
&\qquad\qquad\quad
\text{ существует }\la\subset\ga,\text{ такое что }|\la|\leq\tau\text{ и }\cl{\bigcup\la}=X\ \}.
\end{align*}
Тогда 
$iw(X)$ --- \term{$i$-вес},
$w(X)$ --- \term{вес},
$nw(X)$ --- \term{сетевой вес},
$d(X)$ --- \term{плотность},
$hd(X)$ --- \term{наследственная плотность},
$l(X)$ --- \term{число Линделефа},
$hl(X)$ --- \term{наследственное число Линделефа},
$wl(X)$ --- \term{слабое число Линделефа}
пространства $X$. 

Пространства со счетной сетью (т.е.\ пространства со счетным сетевым весом) так же называются \term{космическими пространствами} (cosmic space). Пространства со счетной плотностью называются \term{сепарабельными} пространствами.

Пространство $X$ называется \term{линделефовым пространством}, если $l(X)\leq\om$. 
Пространство $X$ является \term{линделефовым $p$-пространством}, если $X$ можно совершенно отобразить на пространство со счетной базой. Непрерывные образы линделефовых $p$-пространств называют \term{линделефовыми $\Si$-пространствами}.

Линделефовы $\Si$-пространства также называют \term{счетно определяемыми} (countably determined) или \term{$K$-счетно определяемыми} ($K$-countably determined).

Множество $F\subset X$ называется \term{нуль} множеством пространства $X$, если $F=f^{-1}(0)$  для некоторой непрерывной функции на $X$. Дополнения до нуль множеств называется \term{ко-нуль} множествами.

Пространство $X$  называется \term{совершенно $\ka$-нормальным}, если для любого открытого $U\subset X$ замыкание $\cl U$ является нуль множеством.

Множество $Y\subset X$ \term{$z$-вложено} в $X$, если для любого нуль множества $Q$ множества $Y$ существует  нуль множество $F$ пространства $X$, такое что $Q=F\cap Y$.

Непрерывные функции на $X$ обозначим $C(X)$. Пусть $B\subset X$. Обозначим 
\[
C(X|B) = \set{ \pi_B(f) : f\in C(X) } \subset C(B),
\]
где $\pi_B(f)=\restr fB$.
Семейство функций вида 
\[
\de_x\: C(X) \to \R,\ f\mapsto f(x)
\]
для $x\in B$ определяют топологию $\tau_p(B)$ множества $C(X)$ с топологией поточечной сходимости на множестве $B$.
Если $M\subset C(X)$, то будем обозначать через $(M,\tau_p(B))$  подпространство $M$  топологического пространства $(C(X),\tau_p(B))$.
В общем случае пространство $(M,\tau_p(B))$ неотделимо, но мы будем рассматривать такие $M$ и $B$, что пространство $(M,\tau_p(B))$ отделимо.
Пространство $(C(X),\tau_p(X))$ с топологией поточечной сходимости  на всем множестве $X$ обозначим $C_p(X)$. 
Обозначим $C_p(X|B)$ множество $C(X|B)$ с топологией, наследуемой из пространства $C_p(X)$.
Если пространство $(M,\tau_p(B))$ отделимо, то отображение $\pi_B$ гомеоморфно отображает $(M,\tau_p(B))$ на $\pi_B(M)\subset C_p(X|B) \subset C_p(B)$.

Если $S\subset \R^X$, то $S$ \term{разделяет точки} $X$, если для различных $x,y\in X$ существует $f\in S$, такое что $f(x)\neq f(y)$.

\begin{proposition}[\cite{Arhangelskii1989CpBook}]\label{p:defs:1}
Пусть $X$ пространство.
\begin{enumerate}
\item
$d(X)\leq nw(X) \leq w(X)$ и $iw(X)\leq nw(X)$;
\item
если $X$ метризуемое пространство, то $d(X)=nw(X)=w(X)$;
\item
если $Y$ пространство и $X$ непрерывный образ $Y$, то $nw(X)\leq nw(Y)$;
\item
$nw(X)=nw(C_p(X))$;
\item
$d(X)=iw(C_p(X))$, 
$iw(X)=d(C_p(X))$;
\item
$nw(X)=iw(X)$ если $X$ линделефово $\Si$-пространство;
\item
$w(X)=nw(X)=iw(X)$ если $X$ линделефово $p$-пространство, в частности, если $X$ компактно;
\item
$hd(C_p(X))=hl^*(X)$.
\end{enumerate}
\end{proposition}

Счетно-аддитивная мера $\mu$ на $\si$-алгебре $\cB(X)$ борелевских множеств пространства $X$ называется \term{борелевской} мерой.
Обозначим
\begin{align*}
\mu^*(A)  &= \inf\set{\mu(B): A\subset B\in \cB(X)},
&
\mu_*(A)  &= \sup\set{\mu(B): A\supset B\in \cB(X)}
\end{align*}
\term{внешнюю} и \term{внутреннюю} меру множества $A\subset X$.
Множество вероятностных борелевских мер на пространстве $X$ обозначим через $P_b(X)$.
Пусть $\mu\in P_b(X)$.
Вероятностная мера $\mu$ называется \term{радоновской}, если для всякого борелевского множества $B\subset X$ и для всякого $\ep>0$ существует компактное $K\subset B$, такое что $\mu(B\setminus K)<\ep$. 
Множество вероятностных радоновских мер на пространстве $X$ обозначим через $P(X)$.
Мера $\mu$ называется \term{$\tau$-аддитивной}, если
\[
\mu(\bigcup\cU) = \sup \set{ \mu(\bigcup\cV) : \cV\subset \cU,\ |\cV|<\om }
\]
для любого семейства $\cU$ открытых подмножеств пространства $X$. 
Множество вероятностных $\tau$-аддитивных мер на пространстве $X$ обозначим через $P_\tau(X)$. Отметим, что $P(X)\subset P_\tau(X)$. 

Пусть $Y$ пространство и $X\subset Y$.  Множество $P_b(X)$ естественным образом вложено в $P_b(Y)$, мере $\mu\in P_b(X)$ соответствует мера $\tilde\mu\in P_b(Y)$, такая что $\tilde\mu(B)=\mu(B\cap Y)$ для $B\in\cB(Y)$.  Будем считать, что $P_b(X)\subset P_b(Y)$. Если $Y$ компактно, то $P(X)\subset P_\tau(X)\subset P(Y)$. 
Положим $P_\si(X)=\{\mu\in P(\bt X): \mu(C)=0$ для любого компактного множества $C\subset \bt X\setminus X$ типа $G_\de\}$. Меры из $P_\si(X)$ соответствуют вероятностным $\si$-аддитивным мерам на $\si$-алгебре бэровских множеств пространства $X$.

Обозначим через $C^*(X)$ множество непрерывных ограниченных функций на $X$.
Семейство функции вида 
\[
\ph_f\: P_\tau(X) \to \R,\ \mu \mapsto \mu(f) = \int_X f \mathrm{d}\,\mu
\]
для $f\in C^*(X)$ определяют слабую топология пространства $P_\tau(X)$.

\begin{proposition}[\cite{Banakh1995ru}]\label{p:defs:2}
Пусть $Y$ компактное пространство и $X\subset Y$.
Тогда $w(X)=w(P(X))=w(P_\tau(X))$ и 
\begin{align*}
P(X)  &= \set{\mu \in P(Y): \mu_*(X)=1},
&
P_\tau(X)  &= \set{\mu \in P(Y): \mu^*(X)=1}.
\end{align*}
Вложение $P(X)$ и $P_\tau(X)$ в $P(Y)$ является топологическим вложением.
\end{proposition}

Компактные выпуклые подмножества линейных локально выпуклых пространств будем называть \term{выпуклыми компактами}.
Пусть $K$ выпуклый компакт. 

Точка $x\in K$ называется \term{крайней точкой}, если из того, что $y,z\in K$ и $x=\frac{y+z}2$ вытекает, что $x=y=z$.
Множество крайних точек выпуклого компакта $K$ обозначим как $\ext X$. Множество непрерывных аффинных функций на $K$ обозначим как $\af K$.
Пусть $\nmc$ есть $\sup$-норма на $\af K$, $\nm f=\sup_{x\in K} |f(x)|$ и $\afb K=\set{f\in \af K: \nm f \leq 1}$ --- единичный шар в $(\af K,\nmc)$. 
Для $M\subset \af K$ будем обозначать через $(M,\nmc)$ множество $M$ с метрикой (и, соответственно, с топологией), наследуемой из $(\af K,\nmc)$.

Множество $B\subset K$ называется \term{границей}, если максимум любой аффинной непрерывной функции $f\in \af K$ достигается в некоторой точке $b\in B$,  то есть $f(b)=\max_{x\in K}f(x)$.
Множество крайних точек $\ext K$ является границей в $K$.
Так как выпуклая оболочка границы $B$ плотна в $K$, то $(\af K,\tau_p(B))$ отделимое пространство и $(\af K,\tau_p(B))$ вкладывается в $C_p(K|B)\subset C_p(B)$ \cite{MoorsReznichenko2008}.

\begin{proposition}\label{p:defs:3}
Пусть $K$ выпуклый компакт и $B\subset K$ некоторая граница $K$.
\begin{enumerate}
\item
(теорема Рейнуотер--Симонс, \cite{Simons1972}) Если $(f_n)_n\subset \afb K$ сходится к $f\in \afb K$ в $(\af K,\tau_p(B))$, то $(f_n)_n$ сходится к $f$ в $(\af K,\tau_p(K))$.
Другими словами, сходящиеся последовательности в  $(\afb K,\tau_p(B))$ и $(\afb K,\tau_p(K))$  совпадают.
\item
(теорема Пфицнер, \cite{Pfitzner2010}) Если $Y\subset \afb K$ и $(Y,\tau_p(B))$ компактно, то $(Y,\tau_p(K))$ компактно. 
Другими словами, компактные подмножества в $(\afb K,\tau_p(B))$ и $(\afb K,\tau_p(K))$  совпадают.
\end{enumerate}
\end{proposition}

Пусть $\mu\in P(K)$. Точка $x\in K$ называется \term{барицентром} меры $\mu$, если $f(x)=\int_X f \mathrm{d}\,\mu$ для любого $f\in \af K$. Обозначим через $r_K(\mu)$ барицентр меры $\mu$ в выпуклом компакте $K$.
Обозначим через $P_{\max}(K)$ множество \term{максимальных} мер на выпуклом компакте $K$ (см.\ определение и свойства максимальных мер в главе 1, \$4 \cite{Alfsen1971}).

\begin{proposition}\label{p:defs:4}
Пусть $K$ выпуклый компакт.
\begin{enumerate}
\item
(теорема Мильмана, \cite[теорема 1.12.6]{BogachevSmolyanovSobolev2012})
если $C\subset K$ замкнуто и выпуклая оболочка $C$ плотна в $K$, то $\ext K\subset C$;
\item
(теорема Шоке--Бишоп--де Леу, \cite[теорема I.4.8]{Alfsen1971}) $p_K(P_{\max}(K))=K$;
\item
(теорема Бишопа--де Леу, \cite[следствие I.4.12.]{Alfsen1971})
если $\mu\in P_{\max}(K)$ и $F\subset K\setminus \ext K$ компакт типа $G_\de$, то $\mu(F)=0$;
\item
(\cite[предложение 2.38]{LukesMalyNetukaSpurny2010})
отображение $r_K\: P(K) \to K$ является непрерывным аффинным отображением;
\item
(\cite[следствие 3.60]{LukesMalyNetukaSpurny2010}) $P_\tau(\ext X)\subset P_{\max}(K)$.
\end{enumerate}
\end{proposition}


\section{Cеквенциально плотные подмножества функциональных пространств}
\label{sec:asd}

Пусть $X$ пространство. Семейство подмножеств $\cN$ пространства $X$ назовем \term{\cz-сетью}, если для любого ко-нуль множества $U\subset X$ существует не более чем счетное $\cM\subset \cN$, такое что $\bigcup\cM=U$.
Ясно, любая \cz-сеть является сетью.
Положим
\begin{align*}
\nwcz(X) &= \min \{\ |\cN| : \cN \text{является \cz-сетью пространства }X \}+\omega,
\\
\ds(X) &= \min \{\ |S| : S \text{ является секвенциально плотным в }X \}+\omega,
\\
\hds(X) &= \sup \{\ \ds(M): M\subset X  \}.
\end{align*}

В пространстве Фреше--Урысона плотное подмножество секвенциально плотно, поэтому верно следующие утверждение.

\begin{proposition}\label{p:asd:fu}
Если $X$ пространство Фреше--Урысона, то $\ds(X)=d(X)$ и $\hds(X)=hd(X)$.
\end{proposition}

\begin{lemma}\label{l:asd:1}
Пусть $X$ линделефово $\Si$-пространство. Существует линделефово $p$-пространство $Y$, такое что $X$ является непрерывным образом пространства $Y$ и $w(Y)=nw(X)$.
\end{lemma}
\begin{proof}
Существует линделефово $p$-пространство $Z$ и непрерывное сюрьективное отображение $f \: Z\to X$. Существует сепарабельное метризуемое пространство $M$ и непрерывное сюрьективное совершенное отображение $g\: Z\to M$. 
Положим
\[
h= f \diag g\: Z\to X\times M,\ x\mapsto (f(x),g(x)),
\]
$\pi$ проекция $Y=h(Z)\subset X\times M$ на $X$ и $p$ проекция $Y$ на $M$. Так как отображение $g$ совершенно, то отображение $h= f \diag g$ совершенно \cite[теорема 3.7.11]{EngelkingBookRu}. 
Так как $g=p\circ h$ и отображения $g$ и $h$ совершенны, то отображение $p\: Y\to M$ совершенно. Следовательно, пространство $Y$ является линделефовым $p$-пространством. 
Так как $Y\subset X\times M$, то $iw(Y)\leq iw(X\times M)= iw(X)$. 
Так как $X$ линделефово $\Si$-пространство, $Y$ линделефово $p$-пространство и $X$ непрерывный образ $Y$, то 
\[
w(Y)=iw(Y)\leq iw(X)=nw(X)\leq nw(Y)\leq w(Y)
\]
(см.\ предложение \ref{p:defs:1}). Следовательно, $w(Y)=nw(X)$.
\end{proof}

\begin{theorem}\label{t:asd:1}
Пусть $X$ пространство. Тогда
\begin{enumerate}
\item
$nw(X)\leq \nwcz(X)\leq 2^{nw(X)}$;
\item
если $X$ компактно, то $\nwcz(X)=nw(X)=w(X)$;
\item
если $Y$ пространство и $X\subset Y$ $z$-вложено в $Y$, то $\nwcz(X)\leq \nwcz(Y)$;
\item
если $Y$ пространство и $f\: Y\to X$ непрерывное сюрьективное отображение, то $\nwcz(X)\leq \nwcz(Y)$;
\item
если $X$ линделефово, то $\nwcz(X)\leq w(X)$;
\item
если $X$ линделефово $\Si$-пространство, то $\nwcz(X)= nw(X)$;
\item
если $nw(X)=\om$, то $\nwcz(X)=\omega$;
\item
если $X$ плотно в произведении $\prod_{\al\in A}X_\al$  сепарабельных метризуемых не однототочечных пространств и множество $A$ бесконечно, то $\nwcz(X)\leq w(X)=|A|$.
\end{enumerate}
\end{theorem}
\begin{proof}
(1) Так как \cz-сеть является сетью, то $nw(X)\leq \nwcz(X)$. Так как мощность топологии $X$ не превосходит $2^{nw(X)}$, то $\nwcz(X)\leq 2^{nw(X)}$.

(2) 
Так как любое ко-нуль множество в компакте линделефово, то любая база является \cz-сетью. Следовательно, $\nwcz(X)\leq w(X)$.
Так как $X$ компактно, то $w(X)=nw(X)$ \cite{Hodel1984handbook}. Так как $w(X)=nw(X)\leq \nwcz(X)\leq w(X)$, то $\nwcz(X)=nw(X)=w(X)$.

(3) Пусть $\cN$ \cz-сеть в $Y$. Положим $\cL=\set{M\cap X: M\in\cN}$. Достаточно показать, что $\cL$ \cz-сеть в $X$. Пусть $U$ ко-нуль множество в $X$. Так как $X$ $z$-вложено в $Y$, то $U=W\cap X$ для некоторого ко-нуль множества $W$ в $Y$. Так как $\cN$ \cz-сеть в $Y$, то $W=\bigcup \cN'$ для некоторого счетного $\cN'\subset \cN$. Тогда $\cL'=\set{M\cap X: M\in\cN'}\subset \cL$, то $U=\bigcup\cL'$.

(4) Пусть $\cN$ \cz-сеть в $Y$. Тогда  $\cL=\set{f(M): M\in\cN}$ \cz-сеть в $X$. 

(5) Так как любое ко-нуль множество в линделефовым пространстве линделефово, то любая база является \cz-сетью. Следовательно, $\nwcz(X)\leq w(X)$.

(6) Из леммы \ref{l:asd:1} вытекает, что существует линделефово $p$-пространство $Y$, такое что $X$ является непрерывным образом пространства $Y$ и $w(Y)=nw(X)$.
Из пунктов (4) и (5) вытекает, что $\nwcz(X)\leq \nwcz(Y)\leq w(Y)$. Так как $w(Y)=nw(X)$, то $\nwcz(X)\leq nw(X)$.

(7) Счетная сеть является \cz-сетью.

(8) Пусть $K_\al$ компактное метризуемое расширение пространства $X_\al$ для $\al\in K$. Тогда $X$ плотное подмножество $K=\prod_{\al\in A}K_\al$ и $w(K)=|A|$. 
Произведение  метризуемых  пространств является  совершенно $\ka$-нормальным пространством \cite{Shepin1976}, поэтому $K$ совершенно $\ka$-нормально.
Плотные подмножества совершенно $\ka$-нормальных пространств $z$-вложено в них \cite[Lemma 8.4.5]{at2009}, поэтому $X$ $z$-вложено в компакт $K$. Из  (2) и (3) вытекает, что $\nwcz(X)\leq w(K)=|A|$.
\end{proof}

\begin{proposition}\label{p:asd:1}
Пусть $X$ пространство и $\cN$ семейство подмножеств $X$. Следующие условия эквивалентны:
\begin{enumerate}
\item
семейство $\cN$ является \cz-сетью пространства $X$;
\item
для любого $f\in C(X)$ существует $\cP\subset \cN$, $|\cP|\leq\om$, такое что выполняется условие
\begin{enumerate}
\item[$(F_{cz})$]
для каждого $x\in X$ и $\ep>0$ существует $P\in\cP$, для которого $x\in P$ и
\[
\Om(f,P) = \sup_{x_1,x_2\in P} |f(x_1) - f(x_2)| < \ep.
\]
\end{enumerate}
\end{enumerate}
\end{proposition}
\begin{proof}
$(1)\rarr(2)$ Пусть $\cB$ есть счетная база $\R$. Для $U\in\cB$ существует $\cP_U\subset \cN$, $|\cP_U|\leq\om$, такое что $\bigcup\cP_U=U$. Положим $\cP=\bigcup_{U\in\cB}$. Тогда для $\cP$ выполняется условие $(F_{cz})$.

$(2)\rarr(1)$ Пусть $W=f^{-1}(\R\setminus\sset 0)$ ко-нуль множество и $f\in C(X)$. Возьмем $\cP\subset \cN$, $|\cP|\leq\om$, для которого выполняется условие $(F_{cz})$. Положим $\cM=\set{P\in\cP: P\subset W}$. Тогда $W=\bigcup\cM$.
\end{proof}

\begin{theorem}\label{t:asd:2}
Пусть $X$ пространство. Тогда $\hds(C_p(X))\leq \nwcz(X)$.
\end{theorem}
\begin{proof}
Пусть $M\subset C_p(X)$. Надо найти $S\subset M$, такое что $S$ секвенциально плотно в $M$ и $|S|\leq \nwcz(X)$.
Пространство $C_p(X)$ гомеоморфно пространству $C_p(X,(-1,1))$, так что будем считать, что $M\subset C_p(X,(-1,1))$. 

Пусть $\tau=\nwcz(X)$ и $\cN$ есть \cz-сеть пространства $X$, мощность которой не превосходит $\tau$.
Для $P\in\cP$ положим 
\begin{align*}
\F^+_P &\: M\to \R,\ f\mapsto \sup \set{ f(x)): x\in P },
\\
\F^-_P &\: M\to \R,\ f\mapsto \inf \set{ f(x)): x\in P }.
\end{align*}
Пусть $\tp$ есть топология на $M$, порожденная семейством функций $\set{\F^+_P,\F^-_P: \al<\tau}$.
Тогда вес пространства $(M,\tp)$ не превосходит $\tau$. Возьмем $S\subset M$ таким образом, что $|S|\leq\tau$ и $S$ плотно в $M$ относительно топологии $\tp$.

Покажем, что множество $S$ секвенциально плотно в $M$ относительно топологии $C_p(X)$. Пусть $f\in M$. Из предложения \ref{p:asd:1} вытекает, что существует $\cP=\set{P_n:n<\om}\subset \cN$, для которого выполняется условие $(F_{cz})$. Для $n,m\in\om$ положим
\[
W_{n,m} = \set{g\in M:  -\tfrac{1}{2^m}+\F^-_{P_n}(f) < \F^-_{P_n}(g) \leq \F^+_{P_n}(g) <  \tfrac{1}{2^m}+\F^+_{P_n}(f) }.
\]
Множество $W_{n,m}$ является открытой в $(M,\tp)$ окрестностью точки $f$. Пусть $(n_k,m_k)_{k<\om}$ есть нумерация множества $\om\times \om$. Положим $U_k=\bigcap_{i\leq k}W_{n_i,m_i}$ и выберем $f_k\in U_k\cap S$ для $k<\om$. 

Докажем, что последовательность $(f_k)_k$ поточечно сходится к $f$. Пусть $x\in X$ и $\ep>0$. Возьмем $m<\om$, такое что $\frac{1}{2^m}<\frac \ep 2 $. Из $(F_{cz})$ вытекает, что существует $n\in \om$, такое что $x\in P_n$ и $\Om(f,P_n)<\frac \ep 2$. Тогда $|f(x)-g(x)|<\ep$ для любого $g\in W_{n,m}$. Для некоторого $N<\om$, $n=n_N$ и $m=m_N$. Тогда $f_k\in W_{n,m}$ и $|f(x)-f_k(x)|<\ep$ для $k\geq N$.
\end{proof}

\begin{proposition}[Теорема Комфорда--Хагера {\cite{ComfortHager1970}}]\label{p:asd:ch}
Пусть $X$ пространство. Тогда
\[
w(\bt X)^\om = |C(X)|^\om = |C(X)| \leq w(X)^{wl(X)}.
\]
\end{proposition}

\begin{theorem}\label{t:asd:3}
Пусть $X$ пространство. Тогда 
\[
\nwcz(X)\leq w(\bt X)\leq w(\bt X)^\om = \ds(C_p(X))^\om = \nwcz(X)^\om.
\]
\end{theorem}
\begin{proof}
Из теоремы \ref{t:asd:2} вытекает, что $\ds(C_p(X))\leq \nwcz(X)$.
Cуществует $S\subset C_p(X)$, такое что $S$ секвенциально плотно в $C_p(X)$ и $|S|\leq \ds(C_p(X))$.
Так как $S$ секвенциально плотно в $C_p(X)$, то $|C(X)|\leq |S|^\om=\ds(C_p(X))^\om$.
Так как $X$ $C^*$-вложено в $\bt X$, то $X$ $z$-вложено в $\bt X$. Из теоремы \ref{t:asd:1}(2) и (3) вытекает $\nwcz(X)\leq w(\bt X)$. Следовательно, $\nwcz(X)^\om\leq w(\bt X)^\om$.
Теоремы Комфорда--Хагера влечет, что $\leq w(\bt X)^\om=|C(X)|^\om\leq \nwcz(X)^\om$. Получаем $\nwcz(X)\leq w(\bt X)$ и $w(\bt X)^\om = \ds(C_p(X))^\om = \nwcz(X)^\om$.
\end{proof}

Из теорем  \ref{t:asd:1}(5), \ref{t:asd:2} и \ref{t:asd:3} и теоремы Комфорда--Хагера (предложение \ref{p:asd:ch}) вытекают следующие утверждение.

\begin{cor}\label{c:asd:1}
Пусть $X$ линделефово пространство. Тогда $\hds(C_p(X))\leq w(X)$ и 
\[
w(X) \leq w(\bt X) \leq w(\bt X)^\om = w(X)^\om.
\]
\end{cor}

В \cite{Hager1969} было доказано, что $\ds(C_p(X))\leq w(X)$ и $|C(X)|= w(X)^\om$ для линделефово пространства $X$.

Из теорем  \ref{t:asd:1}(6), \ref{t:asd:2} и \ref{t:asd:3}  следующие утверждение.
\begin{cor}\label{c:asd:2}
Пусть $X$ линделефово $\Si$-пространство. Тогда $\hds(C_p(X))\leq nw(X)$ и 
\[
nw(X) \leq w(\bt X) \leq w(\bt X)^\om = nw(X)^\om.
\]
\end{cor}

В \cite{Tkachenko2017weight} было доказано, что $w(\bt X)\leq |C(X)|\leq nw(X)^\om$ для линделефово $\Si$-пространства $X$.


\section{Вес выпуклых компактов и пространство вероятностных мер}
\label{sec:wcps}

\begin{proposition}\label{p:wcps:1}
Пусть $K$ выпуклый компакт. Тогда $w(K)= nw(P_{\max}(K))$.
\end{proposition}
\begin{proof}
Отображение $p_K\: P(K)\to K$ взятия барицентра меры непрерывно (предложение \ref{p:defs:4}(4)).
Из теоремы Шоке--Бишоп--де Леу (предложение \ref{p:defs:4}(2)) вытекает, что $p_K(P_{\max}(K))=K$. 
Следовательно, $w(K)=nw(K)\leq nw(P_{\max}(K))$. Ясно, $nw(P_{\max}(K))\leq w(P(K))=w(K)$.
\end{proof}

\begin{proposition}\label{p:wcps:2}
Пусть $K$ выпуклый компакт. Если $p_K(P_\tau(\ext K))=K$, то $w(K)=w(\ext K)$.
\end{proposition}
\begin{proof}
Отображение $p_K\: P(K)\to K$ взятия барицентра меры непрерывно (предложение \ref{p:defs:4}(4)).
Из теоремы Т.Банаха (предложение \ref{p:defs:2}) вытекает, что $w(P_\tau(\ext K))=w(\ext K)$.
Следовательно, $w(K)=nw(K)\leq w(P_\tau(\ext K))$.
\end{proof}

Отметим, $P_\tau(\ext X)\subset P_{\max}(K)$ для выпуклого компакта $K$.
\begin{theorem}\label{t:wcps:1}
Пусть $K$ выпуклый компакт.
Если $P_{\max}(K)=P_\tau(\ext K)$, то $w(K)=w(\ext K)=nw(\ext K)$.
\end{theorem}
\begin{proof}
Равенство $w(\ext K)=nw(\ext K)$ было доказано в \cite[теорема 3.2]{Reznichenko1990}.
Осталось применить предложения \ref{p:wcps:2} и \ref{p:defs:4}(5).
\end{proof}

Теорема \ref{t:wcps:1} является усилением теоремы 1.2 из \cite{Spurny2010}.

\begin{proposition}\label{p:wcps:3}
Пусть $K$ выпуклый компакт. Если $\ext K\subset B\subset K$ и $B$ линделефово, то $P_{\max}(K)\subset P_\tau(B)$.
\end{proposition}
\begin{proof}
Пусть $\mu\in P_{\max}(K)$. Для доказательства $\mu\in P_\tau(B)$ нужно проверить, что если $C\subset K\setminus B$ компактно, то $\mu(C)=0$. Пусть $\ga=\{ U\subset K\setminus C: U$ ко-нуль множество в $K\}$. Тогда $\bigcup\ga=K\setminus C$. Так как $B$ линделефово, то существует $(U_n)_n\subset\ga$, такое что $B\subset U=\bigcup_n U_n$. Тогда $B\subset U \subset K\setminus C$ и $C\subset G=K\setminus U\subset K\setminus B$. Множество $G$ является нуль множеством. Из теоремы Бишопа--де Леу (предложение \ref{p:defs:4}) вытекает, что $\mu(G)=0$. Следовательно, $\mu(C)=0$.
\end{proof}

Из теоремы \ref{t:wcps:1} и предложения \ref{t:wcps:1} вытекает следующие утверждение.

\begin{theorem}[{\cite[теорема 3.3]{Reznichenko1990}}]\label{t:wcps:2}
Пусть $K$ выпуклый компакт. Если $\ext K$ линделефово, то и $w(K)=w(\ext K)=nw(\ext K)$.
\end{theorem}

Из предложений \ref{p:wcps:1} и \ref{p:wcps:3}  вытекает следующие утверждение.

\begin{theorem}\label{t:wcps:3}
Пусть $K$ выпуклый компакт. Если $\ext K\subset B\subset K$ и $B$ линделефово, то $w(K)=nw(P_\tau(B))$.
\end{theorem}

Пусть $X$ пространство. Положим
\[
zw(X) = \min\{ w(K) : X \text{ $z$-владывается в компакт }K\}.
\]
Кардинал $zw(X)$ называется \term{$z$-весом} пространства $X$. Тогда $w(P_\si(X))\leq zw(X)$ (\cite[теорема 2.4]{BanakhChigogidzeFedorchuk2003}). Из предложения \ref{t:asd:1} (2) и (3) вытекает, что $\nwcz(X)\leq zw(X)$.

\begin{proposition}\label{p:wcps:4}
Пусть $X$ пространство. Тогда $nw(P_\si(X))\leq \nwcz(X)$.
\end{proposition}
\begin{proof}
Пусть $\cN$ есть \cz-сеть пространства $X$ и $|\cN|\leq \nwcz(X)$. Можно считать, что $\cN$ состоит из замкнутых множеств и семейство $\cN$ замкнуто относительно конечных пересечений, то есть $\bigcup\cL\in \cN$ для конечного $\cL\subset\cN$. 
Положим $\cR=\{ W(F,\ep) : F\in \cN,\ \ep\in (0,1)\cap \Q  \}$, где
\[
W(F,\ep) = \{ \mu \in P_\si(X) : \mu(F)>\ep \}.
\]
Предбазу топологии $P_\si(X)$ образуют множества вида $W(U,\ep)$ для ко-нуль множества $U\subset X$ и $\ep>0$ \cite[4.3]{Bogachev2016book}. Покажем, что $\cR$ есть сеть пространства $P_\si(X)$. 
Пусть $\mu\in W(F,\ep)$. 
Пусть $c\in \Q\cap(\ep,\mu(U))$.
Так как $\cN$ \cz-сеть $X$, то существует последовательность $(F_n)_n\subset \cN$, такая что $\bigcup_n F_n=U$. Положим $P_n=\bigcup_{i\leq n} F_i$. Тогда $\mu(U)=\lim_{n\to\infty}\mu(P_n)$ и, следовательно, $\mu(P_m)>c$ для некоторого $m$. Получаем, $\mu\in W(P_m,c)\subset W(U,\ep)$ и $W(P_m,c)\in\cR$.
\end{proof}

Из теоремы \ref{p:wccs:1}(7) и предложения \ref{p:wcps:4} вытекает следующие утверждение.

\begin{cor}\label{c:wcps:1}
Пусть $X$ пространство. 
Если $X$ со счетной сетью, то $P_\si(X)$ со счетной сетью.
\end{cor}

Из теоремы \ref{t:wcps:3} и предложения \ref{p:wcps:4} вытекает следующие утверждение.

\begin{cor}\label{c:wcps:2}
Пусть $K$ выпуклый компакт. Если $\ext K\subset B\subset K$ и $B$ линделефово, то $w(K)\leq \nwcz(B)$.
\end{cor}

В следующем разделе это следствие будет значительно усилено (теорема \ref{t:wccs:2}).


\section{Вес выпуклых компактов и пространство аффинных функций}
\label{sec:wccs}

Пространство $X$ называется \term{$k$-пространством} если для $F\subset X$, $F$ замкнуто в $X$ если и только если $F\cap K$ замкнуто в $X$ для каждого компакта $K\subset X$. Любое метризуемое пространство является $k$-пространством. Пространство $X$ называется \term{$k_\R$-пространством} если для  $f\: X\to \R$,  функция $f$ непрерывна если и только если для каждого компакта $K\subset X$ ограничение $\restr fK$ функции $f$ на $K$ непрерывно. Любое $k$-пространство является $k_\R$-пространством.

Непосредственно проверяется следующие утверждение.
\begin{lemma}\label{l:wccs:1}
Пусть $X$ $k_\R$-пространство. Отображение пространств $f\:X\to Y$ непрерывно если и только если для каждого компакта $K\subset X$ ограничение $\restr fK$ отображения $f$ на $K$ непрерывно.
\end{lemma}

\begin{lemma}\label{l:wccs:2}
Пусть $X$ пространство и $S\subset C_p(X)$ разделяет точки $X$. Тогда $iw(X)\leq|S|$.
Если $X$ компактно, то $w(X)\leq|S|$. Если $Q\subset S$ плотно в $S$, то $Q$ разделяет точки $X$.
\end{lemma}
\begin{proof}
Так как $S$ разделяет точки $X$, то $X$ непрерывно и иньективно отображается в $\R^S$
\cite[Предложение 0.5.4]{Arhangelskii1989CpBook}. Если $X$ компактно, то $w(X)=iw(X)$ (предложение \ref{p:defs:1}). Покажем, что $Q$ разделяет точки $X$. Пусть $x,y\in X$ и $x\neq y$. Тогда $f(x)\neq f(y)$ для некоторого $f\in S$. Пусть $\ep=\frac{|f(x)-f(y)|}2$. Множество $U=\{g\in C_p(X): |f(x)-g(x)|<\ep$ и $|f(y)-g(y)|<\ep\}$ является открытой окрестностью функции $f$ в $C_p(X)$. Пусть $g\in U\cap Q$. Тогда $g(x)\neq g(y)$.
\end{proof}

\begin{proposition}\label{p:wccs:1}
Пусть $K$ выпуклый компакт, $B\subset K$ граница $K$, $A_1=(\afb K,\tau_p(B))$, $A_1'=(\afb K,\tau_p(K))$  и $\tau$ кардинал.
Следующие условия эквивалентны:
\begin{enumerate}
\item
$w(K)\leq \tau$;
\item
$d(A_1')\leq \tau$;
\item
$\ds(A_1)\leq \tau$;
\item
$A_1$ является непрерывным образом метризуемого пространства $X$, для которого $d(X)\leq\tau$;
\item
$A_1$ является непрерывным образом $k_\R$-пространства $X$, для которого $d(X)\leq\tau$.
\end{enumerate}
\end{proposition}
\begin{proof}
$(1)\rarr(2)$ Пространство $A_1'$ является подпространством пространства $C_p(K)$.
Тогда $nw(C_p(K))=nw(K)=w(K)\leq\tau$ (Предложение \ref{p:defs:1}). Следовательно, $d(A_1')\leq nw(A_1')\leq\tau$.

$(2)\rarr(1)$ Пусть $S\subset A_1'=\cl S$ и $|S|\leq\tau$. 
Так как $\af K$ разделяет точки $K$, то $\afb K$ тоже разделяет точки $K$. Из леммы \ref{l:wccs:2} вытекает $w(K)\leq |S|\leq\tau$.

$(1)\rarr(4)$ Из (1) вытекает, что $d((\af K,\nmc))=w((\af K,\nmc))\leq\tau$. Следовательно,
$d(X)\leq\tau$, где $X=(\afb K,\nmc)$. Пространство $X$ метризуемо и непрерывно отображается на $A_1$.

$(4)\rarr(3)$ Пусть $f\: X\to A_1$ непрерывное сюрьективное отображение. Пусть $S\subset X=\cl S$ и $|S|\leq\tau$. Тогда $S$ секвенциально плотно в $X$ и, следовательно, $f(S)$ секвенциально плотно в $A_1$. Кроме того, $|f(S)|\leq|S|\leq\tau$.

$(3)\rarr(2)$ Пусть $S\subset \af K$ секвенциально плотно в $A_1$ и $|S|\leq\tau$. Из теоремы Рейнуотер--Симонс (предложение \ref{p:defs:3}(1)) вытекает, что $S$ секвенциально плотно в $A_1'$.

$(4)\rarr(5)$ Метризуемое пространство является $k_\R$-пространством.

$(5)\rarr(2)$ Пусть $f\: X\to A_1$ непрерывное сюрьюктивное отображение. Пусть $C\subset X$ компактно. Тогда $f(C)$ компактно. Из теоремы Пфицнера (предложение \ref{p:defs:3}(2)) вытекает, что $f(C)$ компактно в топологии $A_1'$ и имеет ту же топологию, что $f(C)$ в топологии $A_1$. Следовательно, $\restr fC\: C\to A_1'$ непрерывно. Из леммы \ref{l:wccs:1} вытекает, что $f\: X\to A_1'$ непрерывно. Тогда $d(A_1')\leq d(X)\leq\tau$.
\end{proof}

Отметим, что в пункте (3) нельзя заменить $\ds(A_1)\leq \tau$ на $d(A_1)\leq \tau$. Существует выпуклый неметризуемый компакт $K$, для которого $(\af K,\tau_p(\ext K))$ сепарабельно \cite[Example 4.6]{MoorsReznichenko2008}.

Из предложения \ref{p:wccs:1} вытекает следующие утверждение.

\begin{theorem}\label{t:wccs:1}
Пусть $K$ выпуклый компакт и $B\subset K$ граница $K$.
Тогда $w(K)=\ds((\afb K,\tau_p(B)))$.
\end{theorem}

\begin{theorem}\label{t:wccs:1+1}
Пусть $K$ выпуклый компакт и $B\subset K$ граница $K$.
Тогда $w(K)\leq\hds(C_p(B))$.
\end{theorem}
\begin{proof}
Пространство $A_1=(\afb K,\tau_p(B))$ является подпространством пространства $C_p(B)$.
Тогда из теоремы \ref{t:wccs:1} вытекает $w(K)= \ds(A_1)\leq \hds(C_p(B))$.
\end{proof}

\begin{theorem}\label{t:wccs:fu}
Пусть $K$ выпуклый компакт и $B\subset K$ граница $K$.
Если $C_p(B)$ пространство Фреше--Урысона, то $w(K)\leq hl^*(X)$.
\end{theorem}
\begin{proof}
Так как $hd(C_p(X))=hl^*(X)$ (предложение \ref{p:defs:1}), то из предложения \ref{p:asd:fu} вытекает $\hds(C_p(X))\leq hl^*(X)$. Из  теоремы \ref{t:wccs:1+1} вытекает $w(K)\leq hl^*(X)$.
\end{proof}

\begin{theorem}\label{t:wccs:2}
Пусть $K$ выпуклый компакт и $B\subset K$ граница $K$.
Тогда $w(K)\leq \nwcz(B)$.
\end{theorem}
\begin{proof}
Так как $\hds(C_p(B))\leq \nwcz(B)$ (теорема \ref{t:asd:2}), то из теоремы \ref{t:wccs:1+1} вытекает, что $w(K)\leq \nwcz(B)$.
\end{proof}

Из теорем \ref{t:asd:1}(7) и \ref{t:wccs:2} вытекает следующие утверждение.
\begin{theorem}\label{t:wccs:2+1}
Пусть $K$ выпуклый компакт и $B\subset K$ граница $K$.
Если $B$ со счетной сетью, то $K$ метризуемо.
\end{theorem}

Из теорем \ref{t:asd:1}(6) и \ref{t:wccs:2} вытекает следующие утверждение.
\begin{theorem}\label{t:wccs:3}
Пусть $K$ выпуклый компакт и $B\subset K$ граница $K$.
Если $B$ линделефово $\Si$-пространство, то $w(K)\leq nw(B)$.
\end{theorem}

Из теорем \ref{t:asd:1}(5) и \ref{t:wccs:2} вытекает следующие утверждение.
\begin{theorem}[{И.Спурны \cite[Corollary 1.4]{Spurny2010}}]\label{t:wccs:4}
Пусть $K$ выпуклый компакт и $B\subset K$ граница $K$.
Если $B$ линделефово, то $w(K)\leq w(B)$.
\end{theorem}

\begin{theorem}\label{t:wccs:5}
Пусть $K$ выпуклый компакт, $B\subset K$ и выпуклая оболочка $B$ плотна в $K$.
Тогда $w(K)\leq w(\bt B)$.
\end{theorem}
\begin{proof}
Пусть $C=\cl B$. Тогда $w(C)\leq w(\bt B)$ и выпуклая оболочка $C$ плотна в $K$. Из теоремы Мильмана (предложение \ref{p:defs:4}) вытекает $\ext K\subset C$. Следовательно, $C$ граница $K$. Из теоремы Спурны (теорема \ref{t:wccs:4})  вытекает, что $w(K)\leq w(C)\leq w(\bt B)$.
\end{proof}

Из теоремы \ref{t:wccs:5}, теоремы Комфорда--Хагера (предложение \ref{p:asd:ch}), теоремы \ref{t:asd:3} и
следствий \ref{c:asd:1} и \ref{c:asd:2} вытекает следующие утверждение.
\begin{theorem}\label{t:wccs:6}
Пусть $K$ выпуклый компакт, $B\subset K$ и выпуклая оболочка $B$ плотна в $K$.
Тогда
\begin{enumerate}
\item
$w(B) \leq w(K)\leq w(B)^{wl(B)}$;
\item
$\nwcz(B) \leq w(K) \leq \nwcz(B)^\om$;
\item
$w(B) \leq w(K) \leq w(B)^\om$ если $B$ линделефово;
\item
$nw(B)  \leq w(K) \leq  nw(B)^\om$ если $B$ линделефово $\Si$-пространство.
\end{enumerate}
\end{theorem}


\section{Примеры и вопросы}
\label{sec:eq}

Обозначим $\cc=2^\om$ --- мощность континуума.

\begin{example}\label{e:eq:1}
Пусть $\dc$ дискретное множество мощности континуума. Положим
\[
\dcc = \bigcup\{ \clx {\bt \dc}M: M\subset \dc,\ |M|\leq\om \} \subset \bt \dc.
\]
Пространство $\dcc$ локально компактно, счетнокомпактно, $\dcc$ плотно в $\bt\dc$  и $\dc\subset\dcc\subset\bt\dcc=\bt\dc$. 
Локально компактное пространство $\dcc$ является объединением $\cc$ компактов веса $\cc$. Следовательно, $w(\dcc)=\cc$. 

Положим $\kc=P(\bt\dc)$. Тогда $w(\kc)=w(\bt\dc)=2^\cc$ и $\ext\kc=\bt\dc$. Так как $\dcc$ счетнокомпактно и плотно в $\ext\kc$, то $\dcc$ является границей $\kc$. Получаем $w(\kc)=2^\cc>\cc=w(\dcc)$. 
Следовательно, вес выпуклого компакта и вес границы не обязательно совпадают.
Из теоремы \ref{t:wccs:2} вытекает, что $\nwcz(\dcc)=2^\cc>\cc=nw(\dcc)=w(\dcc)$.
\end{example}

\begin{problem}\label{q:eq:1} Пусть $X$ линделефово пространство.
Какие из перечисленных ниже утверждений верны?
\begin{enumerate}
\item
$X$ является непрерывным образом линделефового пространства $Y$, для которого $w(Y)\leq nw(X)$;
\item
$\nwcz(X)= nw(X)$;
\item
$\hds(C_p(X))\leq nw(X)$;
\item
$\ds(C_p(X))\leq nw(X)$;
\item
$w(X)\leq nw(X)^\om$;
\item
$w(X)\leq |X|^\om$;
\item
$w(X)\leq\cc$ если $X$ субметризуемое пространство.
\end{enumerate}
\end{problem}

Отметим, $(1)\rarr(2)\rarr(3)\rarr(4)\rarr(5)\rarr(6)\rarr(7)$ (см.\ разделы \ref{sec:asd}).
Если ответ на пункты (1), (2) или (3) в проблеме \ref{q:eq:1} положителен, то ответ на проблему  \ref{q:2} И.Спурны положителен (см.\ раздел \ref{sec:wccs}).

\begin{example}\label{e:eq:2}
Существует неметризуемый выпуклый компакт $K$ и $\ext K\subset B\subset K$, так что $B$ линделефово и $(\afb K,\tau_p(B))$ сепарабельно \cite[Example 4.5]{MoorsReznichenko2008}. Тогда $B$ линделефова субметризуемая граница неметризуемого выпуклого компакта $K$.
\end{example}

\begin{problem}\label{q:eq:2} 
Пусть $K$ выпуклый компакт и $B\subset K$ линделефова субметризуемая граница. Верно ли, что $w(K)\leq\cc$?
Что, если дополнительно предположить,  $\ext K\subset B$?
\end{problem}

Если ответ на проблему \ref{q:eq:1}(6) положителен, то ответ на проблему \ref{q:eq:3} тоже положителен.

\begin{problem}\label{q:eq:3} 
Пусть $X$ пространство.
Какие из перечисленных ниже утверждений верны?
\begin{enumerate}
\item
$\hds(C_p(X))=\nwcz(X)$;
\item
\cite[Question 4.2]{Osipov2024velichko}
если $\hds(C_p(X))\leq\om$, то $X$ со счетной сетью;
\item
\cite[Question 4.1]{Osipov2024velichko}
если $hd(C_p(X))\leq\om$, то $\hds(C_p(X))\leq\om$;
\item
если $C_p(X)$ пространство Фреше--Урысона, то $nw(X)=hl^*(X)$;
\item
если $C_p(X)$ пространство Фреше--Урысона и $X^n$ наследственно линделефово  для всех натуральных $n$, то $X$ со счетной сетью.
\end{enumerate}
\end{problem}
Напомним, $\hds(C_p(X))\leq\nwcz(X)$ (теорема \ref{t:asd:2}) и $\nwcz(X)\leq\om$ если и только если $X$ со счетной сетью (теорема \ref{t:asd:1}(7)).

Если $\hds(C_p(X))\leq\om$, то $hd(C_p(X))\leq\om$, что эквивалентно тому, что $X^n$ наследственно линделефово  для всех натуральных $n$ \cite[Теорема II.5.33]{Arhangelskii1989CpBook}. 
Не известно примера в ZFC наследственно линделефового в конечных степенях пространства, которое было бы без счетной сети \cite[Question 3.1]{Gruenhage1990opit}.

\begin{problem}\label{q:eq:4} 
Пусть $K$ выпуклый компакт и $B\subset K$ наследственно линделефовая в конечных степенях граница. Верно ли, что $K$ метризуем?
\end{problem}
Если ответ на проблему \ref{q:eq:3}(3) положителен, то ответ на проблему \ref{q:eq:4} также положителен.
Если $C_p(B)$ пространство Фреше--Урысона, то ответ на проблему \ref{q:eq:4} положителен (теорема  \ref{t:wccs:fu}).


\bibliographystyle{elsarticle-num}
\bibliography{ccsb}
\end{document}